\documentclass{article}
\usepackage[final]{graphicx}
\usepackage{psfrag}
\usepackage{amsmath,amsfonts,amssymb,amsxtra,subeqnarray,bm}

\newlength{\extramargin}
\newcommand{\Real}{\ensuremath{{\mathbb{R}}}}

\newcommand{\Complex}{\ensuremath{{\mathbb{C}}}}

\newcommand{\C}{\ensuremath{\mathcal C}}

\newcommand{\one}{\ensuremath{{\mathbf{1}}}}

\newtheorem{theorem}{Theorem}

\newtheorem{lemma}{Lemma}

\newtheorem{definition}{Definition}
\newtheorem{remark}{Remark}

\newenvironment{proof}{\noindent {\bf Proof.}}{\hfill \hspace*{1pt}\hfill$\blacksquare$}

\begin{document}
\title{Synchronization via impulsive deadbeat coupling}
\author{S. Emre Tuna\footnote{The author is with Department of
Electrical and Electronics Engineering, Middle East Technical
University, 06800 Ankara, Turkey. Email: {\tt etuna@metu.edu.tr}}}
\maketitle

\begin{abstract}
For linear networks, where the coupling between the agents takes
place through periodic impulses, a simple method is proposed for
synchronization. It is shown that closing the loop by (normalized)
deadbeat feedback gain produces synchronous behavior if the coupling
strength $\mu$ is large enough. With such choice of control law, in
the limiting case ($\mu\to\infty$) exact synchronization is achieved
after $n$ periods, where $n$ is the order of individual agent
dynamics.
\end{abstract}

\section{Introduction}\label{sec:intro}

\begin{figure}[h]
\begin{center}
\includegraphics[scale=0.55]{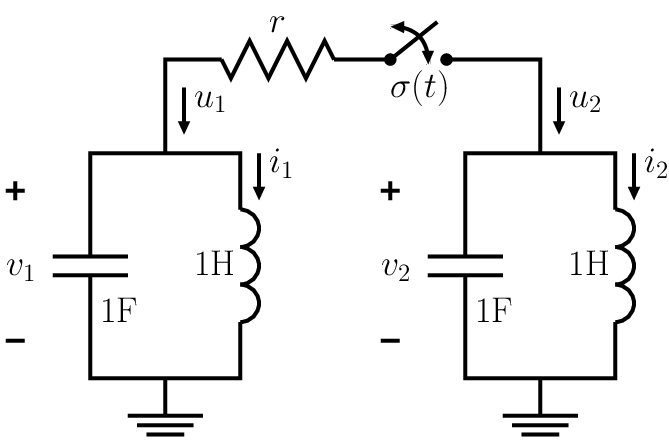}
\caption{Coupled LC oscillators.}\label{fig:motiv}
\end{center}
\end{figure}

\noindent Consider the circuit in Fig.~\ref{fig:motiv} where a pair
of identical LC oscillators are linked by a resistor and a switch.
The position of the switch is represented by the function
$\sigma:\Real\to\{0,\,1\}$. At a given time $t$, $\sigma(t)=1$
indicates that the switch is closed (i.e., the oscillators are
connected) and $\sigma(t)=0$ means the switch is open (i.e., the
oscillators are decoupled). Letting $x_{1}=[v_{1}\ i_{1}]^{\top}$
and $x_{2}=[v_{2}\ i_{2}]^{\top}$ this system enjoys the model
\begin{eqnarray*}
{\dot
x}_{i}&=&\left[\begin{array}{rr}0&-1\\1&0\end{array}\right]x_{i}+\left[\begin{array}{c}1\\
0\end{array}\right]u_{i}\,,\qquad i=1,\,2
\end{eqnarray*}
with the {\em input} currents
\begin{eqnarray*}
u_{1}=-u_{2}=\frac{1}{r}\,\sigma(t)(v_{2}-v_{1})=\frac{1}{r}\,\sigma(t)[1\
\ 0](x_{2}-x_{1})\,.
\end{eqnarray*}
Introducing the matrices
\begin{eqnarray*}
A = \left[\begin{array}{rr}0&-1\\1&0\end{array}\right]\,,\quad B =
\left[\begin{array}{c}1\\ 0\end{array}\right]\,,\quad K = [1\ \
0]\,,\quad\Gamma =
\left[\begin{array}{rr}1&-1\\-1&1\end{array}\right]
\end{eqnarray*}
and defining the overall state $x=[x_{1}^{\top}\
x_{2}^{\top}]^{\top}$ we can represent the network by
\begin{eqnarray*}
{\dot x}=[I\otimes A]x-\frac{1}{r}\,\sigma(t)[\Gamma\otimes BK]x
\end{eqnarray*}
where $I$ is the identity matrix and ``$\otimes\,$'' is the
Kronecker product symbol. Suppose now the switching signal $\sigma$
is a pulse train with period $T$ and pulse width $d$ as shown in
Fig.~\ref{fig:sigma}. Further suppose that the ratio of the pulse
width $d$ to the coupling resistance $r$ is fixed. That is,
$d/r=\mu$ for some $\mu\in\Real$. Then the case $r\ll 1$ can be
approximated by the following impulsive differential equation
\begin{eqnarray}\label{eqn:cast}
{\dot x}=[I\otimes
A]x-\mu\sum_{k=1}^{\infty}\delta(t-kT)[\Gamma\otimes BK]x
\end{eqnarray}
where $\delta$ is the Dirac delta function. It is intuitively clear
that the harmonic oscillators of our simple network will
asymptotically synchronize for any $\mu$ and (almost) any $T$.
However, a remarkable thing happens when $\mu$ is large and the
switching period is set to $T=\frac{\pi}{2}\,$sec. Synchronization
then is (nearly) exact. Namely, $x_{1}(t)=x_{2}(t)$ for $t>2T$. Our
aim in this paper is to provide a formal explanation to this {\em
deadbeat} behavior in a general setting, and thereby obtain a
control design method. Such synchronization method may find
applications in certain networks modeled by impulsive differential
equations.

\begin{figure}[h]
\begin{center}
\includegraphics[scale=0.55]{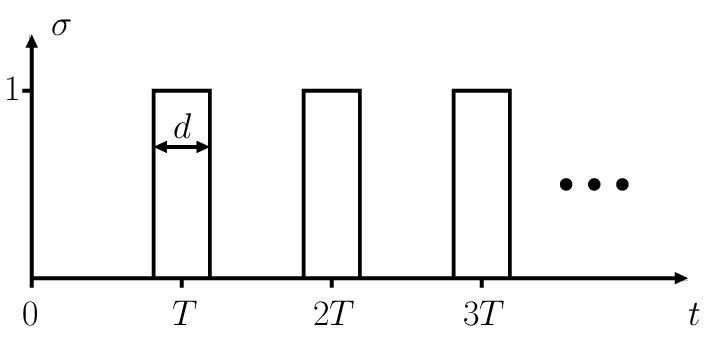}
\caption{Switching signal.}\label{fig:sigma}
\end{center}
\end{figure}

The network dynamics~\eqref{eqn:cast}, which can also be expressed
as \eqref{eqn:onetwo}, make a special case of the so-called
impulsive system \cite{yang01,hespanha08,fiacchini16}. An
interesting area of application for impulsive control theory has
been synchronization over networks. In one of the earliest works on
the subject \cite{yang97}, Yang and Chua establish synchronization
of a pair of chaotic systems via an impulsive control law. In the
same vein, Zhou and coauthors study the collective behavior of
harmonic oscillators subject to impulsive velocity coupling in
\cite{zhou12}. More general setups have also been of interest. For
instance, in \cite{lu10,lu13,he17,wang19} sufficient conditions for
synchronization are presented for a certain class of nonlinear
networks under the assumption that the inner coupling matrix is
positive definite. A leader-follower scheme is employed via
impulsive coupling in \cite{he15}; a method is proposed in
\cite{yang17} by which the agents achieve fixed-time synchronization
through signals communicated from an isolated node. The impulsive
synchronization literature is indeed rich and many other important
results have been reported. For a comprehensive list of recent
research effort we refer the reader to the surveys
\cite{yang19,han20}.

Despite their generality the nonlinear impulsive network models
studied in the related literature do not seem to cover as a special
case the linear impulsive network~\eqref{eqn:cast}, which is what we
intend to investigate in this paper. The coupling we consider being
periodic, it is clear that the behavior of the
network~\eqref{eqn:cast} can be understood through a discrete-time
linear time-invariant (LTI) model. For our analysis therefore we
will not require any tools from impulsive control theory. For the
case with scalar inputs ($u_{i}\in\Real$) and under a standard
connectivity assumption on the communication graph represented by
the Laplacian $\Gamma$, we first show that the units (e.g.,
oscillators) of the network exactly synchronize in finite time if
the matrix $K$ is chosen as the {\em normalized} deadbeat feedback
gain and the coupling strength is arbitrarily large
($\mu\to\infty$). Then, based on this initial observation, we
establish our main result (Theorem~\ref{thm:main}) where we present
a lower bound on $\mu$ for asymptotic synchronization.

\section{Problem statement}

Consider a network of $q$ identical units
\begin{eqnarray}\label{eqn:xi}
{\dot x}_{i}=Ax_{i}+Bu_{i}
\end{eqnarray}
where $x_{i}\in\Real^{n}$ is the state and $u_{i}\in\Real$ the
(scalar) input of the $i$th unit. The matrices $A\in\Real^{n\times
n}$ and $B\in\Real^{n\times 1}$ make a controllable pair $(A,\,B)$.
Available to each unit is a certain signal of the form
\begin{eqnarray}\label{eqn:zi}
z_{i}=\sum_{j=1}^{q}\gamma_{ij}(x_{j}-x_{i})
\end{eqnarray}
with weights $\gamma_{ij}\geq 0$ (we let $\gamma_{ii}=0$). These
signals are used to generate impulsive control inputs
\begin{eqnarray}\label{eqn:ui}
u_{i}=\mu\sum_{k=1}^{\infty}\delta(t-kT)Kz_{i}
\end{eqnarray}
where $\mu>0$ is the coupling strength, $T>0$ is the separation
between two consecutive impulses, and $K\in\Real^{1\times n}$ is the
feedback gain. Combining \eqref{eqn:xi}, \eqref{eqn:zi}, and
\eqref{eqn:ui} produces the closed-loop network dynamics
\begin{eqnarray}\label{eqn:closedloop}
{\dot
x}_{i}=Ax_{i}+\mu\sum_{k=1}^{\infty}\delta(t-kT)BK\left(\sum_{j=1}^{q}\gamma_{ij}(x_{j}-x_{i})\right)\,,\qquad
i=1,\,2,\,\ldots,\,q\,.
\end{eqnarray}
The associated Laplacian $\Gamma\in\Real^{q\times q}$ of this
network reads
\begin{eqnarray}\label{eqn:lap}
\Gamma=\left[\begin{array}{cccc}\sum_{j}\gamma_{1j}&
-\gamma_{12}&\cdots&-\gamma_{1q}\\
-\gamma_{21}&\sum_{j}\gamma_{2j}&\cdots&-\gamma_{2q}\\
\vdots&\vdots&\ddots&\vdots
\\
-\gamma_{q1}&-\gamma_{q2}&\cdots&\sum_{j}\gamma_{qj}\\\end{array}\right]\,.
\end{eqnarray}

\noindent {\bf Standing assumptions.} We make the usual connectivity
assumption that the directed graph of the
network~\eqref{eqn:closedloop} contains a spanning tree. This is
equivalent to saying that $\Gamma$ has a simple eigenvalue at the
origin and the remaining $q-1$ eigenvalues have strictly positive
real parts; see \cite{bullo09}. We further assume that the pair
$(e^{AT},\,B)$ is controllable. Note that this assumption is
satisfied for almost all $T$ since $(A,\,B)$ is controllable.

In this paper we study the collective behavior of the
network~\eqref{eqn:closedloop} for the specific feedback gain $K$
chosen as $K=Ge^{-AT}$ where $G$ is the deadbeat gain for the pair
$(e^{AT},\,B)$, meaning the matrix $\left[e^{AT}-BG\right]$ is
nilpotent. In other words, our goal here is to analyze the impulsive
dynamics~\eqref{eqn:closedloop} under the constraint
\begin{eqnarray}\label{eqn:constraint}
\left[(I-BK)e^{AT}\right]^n=0\,.
\end{eqnarray}
We note that the example network in Section~\ref{sec:intro}
satisfies the condition~\eqref{eqn:constraint} for
$T=\frac{\pi}{2}$.

In the following section we present three lemmas. Those preliminary
results will be invoked later when we set out to establish our main
theorem.

\section{Preliminaries}

Consider the network~\eqref{eqn:closedloop} whose Laplacian $\Gamma$
is given in \eqref{eqn:lap}. Let
$\lambda_{1},\,\lambda_{2},\,\ldots,\,\lambda_{q}$ be the
eigenvalues of $\Gamma$ ordered as $0=\lambda_{1}<{\rm
Re}(\lambda_{2})\leq\cdots\leq{\rm Re}(\lambda_{q})$. Clearly,
$\Gamma\one=0$, where $\one\in\Real^{q}$ is the vector of all ones.
That is, $\one$ is the eigenvector for $\lambda_{1}=0$. We let
$\ell\in\Real^{q}$ denote the left eigenvector for $\lambda_{1}$
(i.e., $\ell^{\top}\Gamma=0$) normalized to satisfy
$\ell^{\top}\one=1$. Note that $e^{-\Gamma t}\to\one\ell^{\top}$ as
$t\to\infty$. Given $Q\in\Complex^{n\times n}$, its conjugate
transpose is denoted by $Q^H$ and its (induced) 2-norm by $\|Q\|$.

Let us construct the (square) controllability matrix $\C=\left[B\ \
e^{AT}B\ \cdots\ e^{A(n-1)T}B\right]$ which is nonsingular since
$(e^{AT},\,B)$ is controllable. Recall that $G\in\Real^{1\times n}$
denotes the deadbeat gain for the pair $(e^{AT},\,B)$. It turns out
that $G$ enjoys the closed-form expression
$G=e_{n}^{\top}\C^{-1}e^{AnT}$ where the unit vector
$e_{n}\in\Real^{n}$ is the $n$th column of the identity matrix; see
\cite{tuna15}. Then, since $K=Ge^{-AT}$, we have
$K=e_{n}^{\top}\C^{-1}e^{A(n-1)T}$. This allows us to claim the
following.

\begin{lemma}\label{lem:KB}
The constraint~\eqref{eqn:constraint} implies $KB=1$.
\end{lemma}

\begin{proof}
$KB=e_{n}^{\top}\C^{-1}e^{A(n-1)T}B=e_{n}^{\top}\C^{-1}\left[B\ \
e^{AT}B\ \cdots\ e^{A(n-1)T}B\right]e_{n}=e_{n}^{\top}\C^{-1}\C
e_{n}=e_{n}^{\top}e_{n}=1.$
\end{proof}
\vspace{0.12in}

\noindent {\bf Alternative proof.} Note that \eqref{eqn:constraint}
implies (by Cayley-Hamilton Theorem) that the characteristic
polynomial of the matrix $\left[(I-BK)e^{AT}\right]$ is
$d(s)=s^{n}$. That is, all the eigenvalues are located at the
origin. Let $w\in\Real^{n}$ be a left eigenvector of the matrix,
i.e., $w^{\top}\left[(I-BK)e^{AT}\right]=0$. This implies, since
$e^{AT}$ is nonsingular, $w^{\top}(I-BK)=0$. Solving the last
equation for $K$ yields $K=(w^{\top}B)^{-1}w^{\top}$. Then
$KB=(w^{\top}B)^{-1}w^{\top}B=1$. \hfill\null\hfill$\blacksquare$
\vspace{0.12in}

The next result is evident. Still we provide a demonstration for the
sake of completeness.

\begin{lemma}\label{lem:delta}
Given $M\in\Real^{m\times m}$ and $\tau\in\Real$, let
$y:\Real\to\Real^{m}$ be a solution to the impulsive differential
equation ${\dot y}=\delta(t-\tau)My$. Then
$y(\tau^{+})=e^{M}y(\tau^{-})$.
\end{lemma}

\begin{proof}
Without loss of generality let $\tau=0$. For $\varepsilon>0$ let the
pulse function $p_{\varepsilon}:\Real\to\Real$ be defined as
\begin{eqnarray*}
p_{\varepsilon}(t)=\left\{\begin{array}{cl}\displaystyle\frac{1}{2\varepsilon}
&\mbox{for}\ \displaystyle
t\in\left[-\varepsilon,\,\varepsilon\right] \vspace{0.1in}\\
0&\mbox{elsewhere}\,.\end{array}\right.
\end{eqnarray*}
Note that $\lim_{\varepsilon\to 0}p_{\varepsilon}(t)=\delta(t)$.
Now, given some $y(\cdot)$ satisfying ${\dot y}=\delta(t)My$ let
$y_{\varepsilon}(\cdot)$ be {\em the} solution to the approximate
differential equation ${\dot
y}_{\varepsilon}=p_{\varepsilon}(t)My_{\varepsilon}$ starting from
the initial condition $y_{\varepsilon}(-\varepsilon)=y(0^{-})$. It
is easy to see that
$y_{\varepsilon}(\varepsilon)=e^{M}y_{\varepsilon}(-\varepsilon)=e^{M}y(0^{-})$.
Therefore
\begin{eqnarray*}
y(0^{+})=\lim_{\varepsilon\to 0}y_{\varepsilon}(\varepsilon)
=e^{M}y(0^{-})
\end{eqnarray*}
which was to be shown.
\end{proof}
\vspace{0.12in}

The following is a result about the {\em Schur decomposition} of the
matrix $\left[(I-BK)e^{AT}\right]$. We will need it for our main
theorem.

\begin{lemma}
There exists a strictly upper triangular $N\in\Complex^{n\times n}$
satisfying
\begin{eqnarray}\label{eqn:schur}
Q^{H}\left[(I-BK)e^{AT}\right]Q=N
\end{eqnarray}
for some unitary $Q\in\Complex^{n\times n}$.
\end{lemma}

\begin{proof}
Recall that all the eigenvalues of $\left[(I-BK)e^{AT}\right]$ are
at the origin by \eqref{eqn:constraint}. The result then follows by
\cite[Thm.~7.1.3]{golub96}.
\end{proof}
\vspace{0.12in}

We end this section with the definition of a synchronous network.

\begin{definition}
The network~\eqref{eqn:closedloop} is said to be {\em synchronous}
if $\|x_{i}(t)-x_{j}(t)\|\to 0$ as $t\to\infty$ for all pairs
$(i,\,j)$ and all initial conditions.
\end{definition}

\section{Main result}

In this section we analyze the collective behavior of the
network~\eqref{eqn:closedloop} under the
condition~\eqref{eqn:constraint} with respect to the coupling
strength $\mu$. The (continuous-time) system~\eqref{eqn:closedloop}
under study is linear time-varying, yet periodic. This means we can
represent it by an LTI model in discrete time. We now obtain this
model. Let $x=[x_{1}^{\top}\ x_{2}^{\top}\ \cdots\
x_{q}^{\top}]^{\top}$. Then \eqref{eqn:closedloop} can be cast into
\eqref{eqn:cast} which yields for $k=1,\,2,\,\ldots$
\begin{subeqnarray}\label{eqn:onetwo}
{\dot x}(t)&=&[I\otimes A]x(t)\,, \quad t\neq kT \\
x(kT^{+})&=&e^{-\mu[\Gamma\otimes BK]}x(kT^{-})
\end{subeqnarray}
where (\ref{eqn:onetwo}b) follows from Lemma~\ref{lem:delta}.
Observe that (\ref{eqn:onetwo}a) allows us to write
$x(kT^{-})=e^{T[I\otimes A]}x((k-1)T^{+})$. Combining this with
(\ref{eqn:onetwo}b) we obtain
\begin{eqnarray}\label{eqn:three}
x(kT^{+})=e^{-\mu[\Gamma\otimes BK]}e^{T[I\otimes
A]}x((k-1)T^{+})\,.
\end{eqnarray}

\begin{lemma}\label{lem:workable} We can write
$e^{-\mu[\Gamma\otimes BK]}e^{T[I\otimes
A]}=\left[e^{-\Gamma\mu}\otimes
e^{AT}\right]+\left[(I-e^{-\Gamma\mu})\otimes(I-BK)e^{AT}\right]$.
\end{lemma}

\begin{proof}
Recall $KB=1$ by Lemma~\ref{lem:KB}. This at once yields
$(BK)^{k}=BK$ for $k=1,\,2,\,\ldots$ (Namely, $[BK]$ is a projection
matrix.) Using the series definition of the matrix exponential we
can then write
\begin{eqnarray*}
e^{-\mu[\Gamma\otimes BK]}
&=& e^{[-\Gamma\mu\otimes BK]}\\
&=& I+[-\Gamma\mu\otimes BK]+[-\Gamma\mu\otimes
BK]^{2}+[-\Gamma\mu\otimes BK]^{3}+\cdots\\
&=& I+[-\Gamma\mu\otimes BK]+\left[(-\Gamma\mu)^{2}\otimes
(BK)^{2}\right]+\left[(-\Gamma\mu)^{3}\otimes
(BK)^{3}\right]+\cdots\\
&=& I+\left[-\Gamma\mu\otimes
BK\right]+\left[(-\Gamma\mu)^{2}\otimes
BK\right]+\left[(-\Gamma\mu)^{3}\otimes
BK\right]+\cdots\\
&=&
I+\left[(-\Gamma\mu+(-\Gamma\mu)^{2}+(-\Gamma\mu)^{3}+\cdots)\otimes
BK\right]\\
&=&I+\left[(e^{-\Gamma\mu}-I)\otimes BK\right]\,.
\end{eqnarray*}
And it is easy to check that
\begin{eqnarray*}
I+\left[(e^{-\Gamma\mu}-I)\otimes
BK\right]=\left[e^{-\Gamma\mu}\otimes
I\right]+\left[(I-e^{-\Gamma\mu})\otimes(I-BK)\right]\,.
\end{eqnarray*}
Since $e^{T[I\otimes A]}=e^{[I\otimes AT]}=\left[I\otimes
e^{AT}\right]$ it follows that
\begin{eqnarray*}
e^{-\mu[\Gamma\otimes BK]}e^{T[I\otimes
A]}&=&\left(\left[e^{-\Gamma\mu}\otimes
I\right]+\left[(I-e^{-\Gamma\mu})\otimes(I-BK)\right]\right)\left[I\otimes
e^{AT}\right]\\
&=&\left[e^{-\Gamma\mu}\otimes
e^{AT}\right]+\left[(I-e^{-\Gamma\mu})\otimes(I-BK)e^{AT}\right]
\end{eqnarray*}
which was to be shown.
\end{proof}
\vspace{0.12in}

For a cleaner presentation we introduce the notation
$x_{i}[k]=x_{i}(kT^{+})$ with $x_{i}[0]=x_{i}(0)$. Then by
\eqref{eqn:three} and Lemma~\ref{lem:workable} we obtain
\begin{eqnarray}\label{eqn:DT}
x[k+1]=\Big(\left[e^{-\Gamma\mu}\otimes
e^{AT}\right]+\left[(I-e^{-\Gamma\mu})\otimes(I-BK)e^{AT}\right]\Big)x[k]
\end{eqnarray}
for $k=0,\,1,\,\ldots$ This is the discrete-time LTI representation
of the network~\eqref{eqn:closedloop} when \eqref{eqn:constraint}
holds.

\begin{remark}\label{rem:one}
Note that if $\|x_{i}[k]-x_{j}[k]\|\to 0$ as $k\to\infty$ for all
pairs $(i,\,j)$ and all initial conditions then the
network~\eqref{eqn:closedloop} is synchronous. Moreover, if for some
$k$ we have $x_{i}[k]=x_{j}[k]$ for all $(i,\,j)$ then
$x_{i}(t)=x_{j}(t)$ for all $(i,\,j)$ and $t>kT$.
\end{remark}

Below is our main result.

\begin{theorem}\label{thm:main}
Consider the network~\eqref{eqn:closedloop} under the
condition~\eqref{eqn:constraint}. Suppose
\begin{eqnarray}\label{eqn:mu}
\mu>\frac{1}{{\rm Re}(\lambda_{2})}\ln\left(
\left\|BKe^{AT}\right\|\times\sum_{k=0}^{n-1}\|N\|^{k}\right)
\end{eqnarray}
where $N$ comes from the Schur decomposition~\eqref{eqn:schur}. Then
the network is synchronous.
\end{theorem}

Prior to demonstrating Theorem~\ref{thm:main} we find it worthwhile
to present its adumbration.

\begin{theorem}\label{thm:adumbration}
Consider the network~\eqref{eqn:closedloop} under the
condition~\eqref{eqn:constraint}. Suppose $\mu=\infty$. Then the
solutions satisfy $x_{i}(t)=x_{j}(t)$ for all $(i,\,j)$ and $t>nT$.
\end{theorem}

\begin{proof}
Consider \eqref{eqn:DT}. Recall $e^{-\Gamma \mu}\to\one\ell^{\top}$
as $\mu\to\infty$ where $\ell^{\top}\one=1$. Hence
\begin{eqnarray}\label{eqn:TVG}
x[k+1]=\Big(\left[\one\ell^{\top}\otimes
e^{AT}\right]+\left[(I-\one\ell^{\top})\otimes(I-BK)e^{AT}\right]\Big)x[k]\,.
\end{eqnarray}
Observe that
$\one\ell^{\top}(I-\one\ell^{\top})=(I-\one\ell^{\top})\one\ell^{\top}=0$.
Moreover, $(\one\ell^{\top})^{k}=\one\ell^{\top}$ and
$(I-\one\ell^{\top})^{k}=I-\one\ell^{\top}$ for $k\geq 1$. Therefore
\begin{eqnarray*}
x[k]&=&\Big(\left[\one\ell^{\top}\otimes
e^{AT}\right]+\left[(I-\one\ell^{\top})\otimes(I-BK)e^{AT}\right]\Big)^{k}x[0]\\
&=&\Big(\left[\one\ell^{\top}\otimes
e^{AT}\right]^{k}+\left[(I-\one\ell^{\top})\otimes(I-BK)e^{AT}\right]^{k}\Big)x[0]\\
&=&\Big(\left[\one\ell^{\top}\otimes
e^{AkT}\right]+\left[(I-\one\ell^{\top})\otimes\left[(I-BK)e^{AT}\right]^{k}\right]\Big)x[0]\,.
\end{eqnarray*}
Then, thanks to \eqref{eqn:constraint}, setting $k=n$ yields
\begin{eqnarray*}
x[n]=\left[\one\ell^{\top}\otimes e^{AnT}\right]x[0]=\one\otimes
\left(e^{AnT}\left[\ell^{\top}\otimes I\right]x[0]\right)
\end{eqnarray*}
which means $x_{i}[n]=e^{AnT}\left[\ell^{\top}\otimes I\right]x[0]$
for all $i$. The result follows by Remark~\ref{rem:one}.
\end{proof}
\vspace{0.12in}

An interesting observation is in order here: The deadbeat character
of synchronization established in Theorem~\ref{thm:adumbration} is
not lost under time-varying communication topology (as long as
$\mu=\infty$). That is, even under a time-varying Laplacian matrix,
the individual solutions still converge to a common trajectory in
finite time, provided that $\Gamma$ is continuous and the associated
graph admits a spanning tree at times $t=T,\,2T,\,\ldots$ Here is
why. Under time-varying $\Gamma$, the difference
equation~\eqref{eqn:TVG} gets generalized to
\begin{eqnarray*}
x[k+1]=\Big(\left[\one\ell_{k}^{\top}\otimes
e^{AT}\right]+\left[(I-\one\ell_{k}^{\top})\otimes(I-BK)e^{AT}\right]\Big)x[k]
\end{eqnarray*}
where $\ell_{k}$ is the left eigenvector of $\Gamma$ at $t=kT$ for
the eigenvalue at the origin, with the property
$\ell_{k}^{\top}\one=1$. By inspection one realizes
$\left[\one\ell_{k}^{\top}\otimes
e^{AT}\right]+\left[(I-\one\ell_{k}^{\top})\otimes(I-BK)e^{AT}\right]=\left[I\otimes
(I-BK)e^{AT}\right]+\left[\one\ell_{k}^{\top}\otimes
BKe^{AT}\right]$. This allows us to write
\begin{eqnarray*}
x[k+1]&=&\Big(\left[I\otimes
(I-BK)e^{AT}\right]+\left[\one\ell_{k}^{\top}\otimes
BKe^{AT}\right]\Big)x[k]\\
&=&\left[I\otimes
(I-BK)e^{AT}\right]x[k]+\one\otimes\Big(\left[\ell_{k}^{\top}\otimes
BKe^{AT}\right]x[k]\Big)
\end{eqnarray*}
which implies for all $i$
\begin{eqnarray*}
x_{i}[k+1]=\left[(I-BK)e^{AT}\right]x_{i}[k]+\left[\ell_{k}^{\top}\otimes
BKe^{AT}\right]x[k]\,.
\end{eqnarray*}
Since for all $i$ the second term $\left[\ell_{k}^{\top}\otimes
BKe^{AT}\right]x[k]$ is the same, for any pair $(i,\,j)$ the error
$e_{ij}=x_{i}-x_{j}$ must evolve through
\begin{eqnarray*}
e_{ij}[k+1]=\left[(I-BK)e^{AT}\right]e_{ij}[k]
\end{eqnarray*}
which, by \eqref{eqn:constraint}, yields $e_{ij}[n]=0$. That is,
$x_{i}[n]=x_{j}[n]$ and the deadbeat behavior is hence preserved.
\vspace{0.12in}

\noindent{\bf Proof of Theorem~\ref{thm:main}.} Let
$U\in\Complex^{q\times(q-1)}$ be such that
\begin{eqnarray*}
\Gamma[\one\ \ U]=[\one\ \ U]\left[\begin{array}{cc}
0&0\\
0&\Lambda
\end{array}\right]
\end{eqnarray*}
where $\Lambda\in\Complex^{(q-1)\times(q-1)}$ is an upper block
triangular matrix whose diagonal entries are the nonzero eigenvalues
of $\Gamma$. Namely,
\begin{eqnarray*}
\Lambda=\left[\begin{array}{cccc}
\lambda_{2}&\ast&\cdots&\ast\\
0&\lambda_{3}&\cdots&\ast\\
\vdots&\vdots&\ddots&\vdots\\
0&0&\cdots&\lambda_{q}
\end{array}\right]\,.
\end{eqnarray*}
Then let $V\in\Complex^{q\times(q-1)}$ be such that $[\one\ \
U]^{-1}=[\ell\ \ V]^{\top}$. Note that this implies
$V^{\top}\Gamma=\Lambda V^{\top}$. Introduce the auxiliary variables
$\eta\in\Real^{n}$ and $\xi\in\Complex^{(q-1)n}$ as
$\eta=\left[\ell^{\top}\otimes I\right]x$ and
$\xi=\left[V^{\top}\otimes I\right]x$. Since
$\one\ell^{\top}+UV^{\top}=[\one\ \ U][\ell\ \ V]^{\top}=I$, we can
decompose $x$ as
\begin{eqnarray*}
x=[\one\otimes I]\eta+[U\otimes I]\xi\,.
\end{eqnarray*}
This decomposition tells us that if $\xi[k]\to 0$ as $k\to\infty$
then $\|x_{i}[k]-\eta[k]\|\to 0$ for all $i$; meaning the
network~\eqref{eqn:closedloop} is synchronous by
Remark~\ref{rem:one}. Now we establish that the
condition~\eqref{eqn:mu} guarantees the convergence $\xi[k]\to 0$.
Observe that $\left[e^{-\Gamma\mu}\otimes
e^{AT}\right]+\left[(I-e^{-\Gamma\mu})\otimes(I-BK)e^{AT}\right]=\left[I\otimes
(I-BK)e^{AT}\right]+\left[e^{-\Gamma\mu}\otimes BKe^{AT}\right]$.
Also note that $V^{\top}\Gamma=\Lambda V^{\top}$ implies
$V^{\top}e^{-\Gamma\mu}=e^{-\Lambda\mu}V^{\top}$. Hence by
\eqref{eqn:DT} we can write
\begin{eqnarray}\label{eqn:ksi}
\xi[k+1]&=&\left[V^{\top}\otimes I\right]x[k+1]\nonumber\\
&=&\left[V^{\top}\otimes I\right]\Big(\left[e^{-\Gamma\mu}\otimes
e^{AT}\right]+\left[(I-e^{-\Gamma\mu})\otimes(I-BK)e^{AT}\right]\Big)x[k]\nonumber\\
&=&\Big(\left[V^{\top}\otimes
(I-BK)e^{AT}\right]+\left[V^{\top}e^{-\Gamma\mu}\otimes BKe^{AT}\right]\Big)x[k]\nonumber\\
&=&\Big(\left[V^{\top}\otimes
(I-BK)e^{AT}\right]+\left[e^{-\Lambda\mu}V^{\top}\otimes BKe^{AT}\right]\Big)x[k]\nonumber\\
&=&\Big(\left[I\otimes
(I-BK)e^{AT}\right]+\left[e^{-\Lambda\mu}\otimes BKe^{AT}\right]\Big)\left[V^{\top}\otimes I\right]x[k]\nonumber\\
&=&\underbrace{\Big(\left[I\otimes
(I-BK)e^{AT}\right]+\left[e^{-\Lambda\mu}\otimes
BKe^{AT}\right]\Big)}_{\displaystyle\Phi}\xi[k]\,.
\end{eqnarray}
In the light of \eqref{eqn:ksi} the convergence $\xi[k]\to 0$ is
equivalent to that all the eigenvalues of the matrix $\Phi$ are in
the open unit disc. Since $\Lambda$ is upper triangular, so is
$e^{-\Lambda\mu}$ with diagonal entries
$e^{-\lambda_{2}\mu},\,e^{-\lambda_{3}\mu},\,\ldots,\,e^{-\lambda_{q}\mu}$.
Thus the matrix $\Phi$ is upper block triangular with diagonal block
entries
$\left[(I-BK)e^{AT}+e^{-\lambda_{i}\mu}BKe^{AT}\right]=:D_{i}$ for
$i=2,\,3,\,\ldots,\,q$. Let now $\varphi\in\Complex$ be an arbitrary
eigenvalue of $\Phi$. Then $\varphi$ must be an eigenvalue of some
block $D_{i}$. Recall that all the eigenvalues of
$\left[(I-BK)e^{AT}\right]$ are zero due to \eqref{eqn:constraint}.
Combining this with the result \cite[Thm.~7.2.3]{golub96} from
perturbation theory we reach to the following implication
\begin{eqnarray*}
\left\|e^{-\lambda_{i}\mu}BKe^{AT}\right\|\times\sum_{k=0}^{n-1}\|N\|^{k}<1
\implies |\varphi|<1
\end{eqnarray*}
where $N$ satisfies \eqref{eqn:schur}. It is easy to check that the
left-hand side of the implication is satisfied for all
$i\in\{2,\,3,\,\ldots,\,q\}$ if $\mu$ satisfies \eqref{eqn:mu}
thanks to the ordering ${\rm Re}(\lambda_{2})\leq\cdots\leq{\rm
Re}(\lambda_{q})$. Hence the result. \hfill\null\hfill$\blacksquare$

\section{Conclusion}

In this paper we investigated the linear network
\eqref{eqn:closedloop} where the coupling between the units is
realized through periodic impulses. We showed that if the state
feedback gain $K$ is chosen to satisfy the {\em deadbeat} condition
\eqref{eqn:constraint} then for large enough coupling strength $\mu$
the units throughout the network asymptotically synchronize. The
identity $KB=1$, which comes for free under the deadbeat condition,
renders $[BK]$ a projection matrix. This observation was essentially
what made our analysis here possible. This is also the flower pot
under which is the key for generalization. Namely, if one replaces
the condition~\eqref{eqn:constraint} with the weaker assumption:
\begin{center}
{\em All the eigenvalues of $\left[(I-BK)e^{AT}\right]$ are in the
open unit disc and $[BK]$ is a projection matrix.}
\end{center}
then it should still hold that synchronization is achieved for $\mu$
large enough.

\bibliographystyle{plain}
\bibliography{references}
\end{document}